\documentclass[a4paper,11pt]{amsart}

\usepackage[all]{xy}
\SelectTips{cm}{}

\usepackage{tikz}
\usepackage{tikz-cd}
\usepackage{mathtools}
\usepackage{amsmath,amssymb}
\usepackage{verbatim}
\usepackage[colorlinks, linkcolor=blue,anchorcolor=Periwinkle,
citecolor=blue,urlcolor=Emerald]{hyperref}
\usepackage[french,english]{babel}
\usetikzlibrary{decorations.pathmorphing}

\newcommand{\ko}{\: , \;}

\numberwithin{equation}{subsection}
\newtheorem{theorem}[subsection]{Theorem}
\newtheorem{definition}[subsection]{Definition}
\newtheorem{classification-theorem}[subsection]{Classification Theorem}
\newtheorem{decomposition-theorem}[subsection]{Decomposition Theorem}
\newtheorem{proposition-definition}[subsection]{Proposition-Definition}
\newtheorem{definition-proposition}[subsection]{Definition-Proposition}
\newtheorem{example-definition}[subsection]{Example-Definition}
\newtheorem{periodicity-conjecture}[subsection]{Periodicity Conjecture}
\newtheorem{lemma}[subsection]{Lemma}
\newtheorem{proposition}[subsection]{Proposition}
\newtheorem{corollary}[subsection]{Corollary}

\newtheorem{example}[subsection]{Example}
\newtheorem{remark}[subsection]{Remark}

\newtheorem{notation}[subsection]{Notation}

\newtheorem{Definition-Proposition}[subsection]{D\'efinition-Proposition}

\newtheorem*{theorema}{Theorem A}
\newtheorem*{theoremb}{Theorem B}

\newcommand{\reminder}[1]{}

\renewcommand{\mod}{\mathrm{mod}}

\newcommand{\Mod}{\mathrm{Mod}\,}

\newcommand{\tr}{\mathrm{tr}}

\newcommand{\Cone}{\mathrm{Cone}}

\newcommand{\Hqe}{\mathrm{Hqe}}

\newcommand{\pretr}{\mathrm{pretr} }

\newcommand{\iso}{\xrightarrow{_\sim}}

\newcommand{\id}{\mathbf{1}}

\newcommand{\Def}{\mathrm{def}\kern 0.1em}

\newcommand{\D}{\mathcal {D}}
\newcommand{\A}{\mathcal {A}}
\newcommand{\B}{\mathcal {B}}
\newcommand{\C}{\mathcal {C}}
\newcommand{\E}{\mathcal {E}}

\newcommand{\N}{\mathcal {N}}

\newcommand{\T}{\mathcal T}
\renewcommand{\P}{\mathcal P}

\newcommand{\Hom}{\mathrm{Hom}}

\newcommand{\Mor}{\mathrm{Mor}}

\renewcommand{\S}{\mathcal{S}}

\newcommand{\dg}{\mathrm{dg}}

\renewcommand{\phi}{\varphi}

\renewcommand{\tilde}[1]{\widetilde{#1}}

\begin{document}

\date{\today}
\title[Enhanced left triangulated categories]
{Enhanced left triangulated categories}
\author{Xiaofa Chen}
\address{University of Science and Technology of China, Hefei, P.~R.~China}
\email{cxf2011@mail.ustc.edu.cn}
\subjclass[2020]{18G35, 18G25, 18E20, 16E30, 16E45}
\keywords{exact dg category, dg derived category, left stable dg category.}
\maketitle
\begin{abstract}
In this short note, we study dg categories with homotopy kernels, whose homotopy categories are known to admit a natural left triangulated structure. 
Prototypical examples of such dg categories arise as dg quotients of exact dg categories.  
We demonstrate that the stablization of the homotopy category of such a dg category admits a canonical dg enhancement via its bounded derived dg category. 
\end{abstract}
\section{Introduction}
DG categories already appear in the work of Kelly~\cite{Kelly65}. They were applied to matrix problems in the representation theory of finite-dimensional algebras. In the 1990s, Bondal--Kapranov observed that the morphism spaces of the  triangulated categories that appear naturally in algebra are the zeroth cohomology groups of certain complexes. The now basic principle (popularized by Grothendieck) in homological algebra is to consider the complex itself rather than its cohomology. Therefore, they proposed to enhance triangulated categories using dg categories, requiring that the triangulated structure should be compatible with the dg structure. Here, the sequences in the dg categories that enhance the distinguished triangles are essentially the graded-split short exact sequences. 

Left triangulated categories are one-sided analogues to triangulated categories.  A left triangulated category consists of an additive category $\T$ with an endofunctor $\Sigma:\T\rightarrow \T$, and a collection of {\em exact left-triangles}, satisfying certain axioms. 
 They enjoy about half of the usual properties of triangulated categories. Since an exact left-triangle will not induce long exact sequences of Hom spaces (as opposed to a distinguished triangle in a triangulated category), we could not expect using graded-split short exact sequences to enhance it.

On the other hand, based on the notion of {\em homotopy short exact sequence}, the author proposed in \cite{Chen26} the notion of {\em stable dg category} to enhance triangulated categories. A homotopy short exact sequence is at the same time homotopy left exact and homotopy right exact. One could therefore expect to enhance left triangulated categories based on the notion of homotopy left exact sequence. This was shown by Mochizuki in \cite{Mochizuki25a}.
A dg category $\A$ is {\em left stable} (= finitely complete) \cite[Definition 2.1]{Mochizuki25a} provided that it admits homotopy kernels, i.e.~for any morphism $j:B\rightarrow C$ in $Z^0(\A)$, it has a homotopy kernel
\[
\begin{tikzcd}
A\ar[r,"f", tail]\ar[rr,bend right=8ex,"h"swap] & B\ar[r,"j",two heads] &C.
\end{tikzcd}
\]
It was shown \cite[Proposition 2.17]{Mochizuki25a} that if $\A$ is a left stable dg category,  then $H^0(\A)$ carries a canonical left triangulated structure.

 The connective cover $\tau_{\leq 0} (\A)$ of the left stable dg category $\A$ remains left stable, and the homotopy category $H^0(\A)$ remains the same. Therefore, we can always assume that the left stable dg category is connective.
 Our aim in this note is to show that the stablization of the left triangulated category $H^0(\A)$ is enhanced by the bounded derived dg category $\D^b_{\dg}(\A)$.
\begin{theorema}[{=Theorem~\ref{thm:universal}}]
Let $\A$ be a left stable dg category. There is a universal exact morphism $F:\A\rightarrow \D^b_{\dg}(\A)$ in $\Hqe$ from $\A$ to a pretriangulated dg category.
Put $\D^b(\A)=H^0(\D^b_{\dg}(\A))$.
If $\A$ is moreover connective, then the functor $H^0(F):H^0(\A)\rightarrow \D^b(\A)$ identifies with the stablization functor $H^0(\A)\rightarrow \S(H^0(\A))$.
\end{theorema}

Let $\Lambda$ be an artin algebra. It is classical that the ideal quotient $\underline{\mod}\mbox{-}\Lambda$ of the module category $\mod\mbox{-}\Lambda$ modulo the projectives is left triangulated. We provide an enhanced version of this result.
Let $(\B,\mathcal S)$ be a connective exact dg cateogry with enough projectives. We denote by $\P$ the full dg subcategory of $\B$ consisting of the projectives.
\begin{theoremb}[{=Theorem~\ref{thm:dgquotient}}]
The dg quotient $\B/\P$ is left stable.
\end{theoremb}
In summary, for a connective exact dg category $\B$ with enough projectives, we have the following diagram
\[
\begin{tikzcd}
\B\ar[r]\ar[d,rightsquigarrow]&\B/\P\ar[r]\ar[d,rightsquigarrow]&\D^b_{\dg}(\B/\P)\simeq \D^b_{\dg}(\B)/\pretr(\P)\ar[d,rightsquigarrow]\\
H^0(\B)\ar[r]&H^0(\B)/[\P]\ar[r]&\S(H^0(\B)/[\P])\simeq \D^b(\B)/\tr(\P).
\end{tikzcd}
\]

\section{Left stable dg categories}
A dg category $\A$ is {\em connective} provided that $H^{>0}\Hom_{\A}(A,B)=0$ for any pair of objects $A$, $B$ in $\A$.
Throughout this section, let $\A$ be an additive dg category, i.e.~$H^0(\A)$ is an additive category. 

\begin{definition}[{\cite[Definition 2.1]{Mochizuki25a}}]
\label{def:leftstable}
\rm{
A dg category $\A$ is {\em left stable} (=finitely complete) provided that it admits homotopy kernels, i.e.~for any morphism $j:B\rightarrow C$ in $Z^0(\A)$, it has a homotopy kernel
\[
\begin{tikzcd}
A\ar[r,"f", tail]\ar[rr,bend right=8ex,"h"swap] & B\ar[r,"j"] &C.
\end{tikzcd}
\]
}
\end{definition}
It is routine to verify the following.
\begin{proposition}[{\cite[Proposition 2.17]{Mochizuki25a}}]
\label{thm:main}
Let $\A$ be a left stable dg category. Then $H^0(\A)$ carries a canonical left triangulated structure.
\end{proposition}
A morphism $F:\A\rightarrow \B$ in $\Hqe$ from $\A$ to an exact dg  category $(\B,\mathcal S)$ is {\em exact}, provided that $F$ sends homotopy left exact sequences in $\A$ to homotopy short exact sequences in $\mathcal S$.

\begin{theorem}\label{thm:universal}
Let $\A$ be a left stable dg category. There is a universal exact morphism $F:\A\rightarrow \D^b_{\dg}(\A)$ in $\Hqe$ from $\A$ to a pretriangulated dg category.
If $\A$ is moreover connective, then the functor $H^0(F):H^0(\A)\rightarrow \D^b(\A)$ identifies with the stablization functor $H^0(\A)\rightarrow \S(H^0(\A))$.
\end{theorem}

The rest of this subsection is devoted to proving Theorem~\ref{thm:universal}.
Let $\mathcal N$ be the full triangulated subcategory of $\tr(\A)$ generated by the total dg modules $N$ of homotopy left exact sequences 
\[
\begin{tikzcd}
A\ar[r,"f"]\ar[rr,"h"swap,bend right=8ex]&B\ar[r,"j"]&C
\end{tikzcd}. 
\]
So $N$ is defined by the following diagram in $\C_{\dg}(\A)$
\begin{equation}\label{TR4}\tag{$\bigstar$}
\begin{tikzcd}[every label/.append style={font=\tiny}]
A\ar[r,"f"]\ar[d,"\begin{bmatrix}-f\\-h\end{bmatrix}",swap]&B\ar[r,"\begin{bmatrix}0\\1\end{bmatrix}"]\ar[d,equal]&U\ar[rd,"s"red,red]\ar[r,"{[}1{,}\;0{]}"]\ar[d,"{[}-h{,}j{]}",swap]&\Sigma A\ar[d]\\ 
V\ar[r,"{[}-1\ 0{]}"swap]&B\ar[r,"j"swap]&C\ar[r,"\begin{bmatrix}0\\1\end{bmatrix}"swap]\ar[d,"\begin{bmatrix}0\\0\\1\end{bmatrix}",swap]&\Sigma V\ar[d]\\ 
& &N\ar[r,equal]\ar[d]&N\ar[d]\\
&&\Sigma U\ar[r]&\Sigma^{2}A
\end{tikzcd}
\end{equation}
where we omit the symbol $\wedge$ for representable dg modules and where 
\[
s=\begin{bmatrix}0&0\\-1&0\end{bmatrix} \ko U=\Cone(f)\, \ko V=\Sigma^{-1}\Cone(j) \mbox{ and }
N=\Cone([-h,\;j]).
\]

Let $\mathcal N_{\dg}$ be the full dg subcategory of $\pretr(\A)$ consisting of the objects in $\N$. 
Let $F$ be the canonical morphism from $\A$ to the Drinfeld dg quotient $\pretr(\A)/\mathcal N_{\dg}$
(recall that we assume the morphism complexes of $\A$ to be cofibrant hence flat). We consider
$F$ as a morphism in the category $\Hqe$ obtained from the category of small dg categories
by localizing at the class of quasi-equivalences.

\begin{lemma}\label{lem:univer} 
The morphism
$F:\A\rightarrow \pretr(\A)/\mathcal N_{\dg}$ is the universal exact morphism from $\A$ to a pretriangulated dg category.
\end{lemma}
\begin{proof}
Put $\D^b(\A)=H^0(\pretr(A)/\mathcal N_{\dg})$.
We first show that $F$ is exact.
Let $X$ be a homotopy left exact sequence as follows
\[
\begin{tikzcd}
A\ar[r,"f"]\ar[rr,"h"swap,bend right=8ex]&B\ar[r,"j"]&C
\end{tikzcd}.
\]
We consider the following roof $[-f,-h]^{\intercal} \backslash [0,-1]^{\intercal}$: 
\[
\begin{tikzcd}
&V&\\
\Sigma^{-1}C\ar[ru,rightarrow,"{[}0{,}-1{]^{\intercal}}"]&&A\ar[lu,Rightarrow,"{[}-f{,}-h{]^{\intercal}}"{swap}]
\end{tikzcd}
\]
where $V$ is $\Sigma^{-1}$ of the mapping cone of $j:B\rightarrow C$.
We have the following triangle in $\D^b(\A)$:
\[
\begin{tikzcd}
\Sigma^{-1}C\ar[rr,"{[}-f{,}-h{]}^{\intercal} \backslash {[}0{,}-1{]}^{\intercal}"]&&A\ar[r,"f/1"]&B\ar[r,"j/1"]&C.
\end{tikzcd}
\]
So the morphism $F$ is exact.

We show that $F$ is universal.
Let $G:\A\rightarrow \B$ be an exact morphism in $\mathrm{Hqe}$ from $\A$ to a pretriangulated dg category $\B$.
By the universal property of $\A\rightarrow \pretr(\A)$, there exists a unique morphism from $\pretr(\A)$ to $\B$ in $\mathrm{Hqe}$ which extends $G$.
Since $G$ is an exact morphism, $H^0(G)$ sends the totalization $N$ of $X$ to a zero object in $H^0(\B)$.
By the universal property of Drinfeld dg quotient \cite{Tabuada10}, there exists a unique morphism in 
$\Hqe$ from $\pretr(\A)/\mathcal N_{\dg}$ to $\B$.
\end{proof}

In the rest, we assume that the dg category $\A$ is connective. In this case, we have 

\begin{itemize}
\item The left aisle $\D(\A)^{\leq 0}$ is formed by those dg modules whose cohomology is concentrated in non-positive degrees.
\item The heart $\mathcal H$ of the t-structure $(\D(\A)^{\leq 0},\D(\A)^{\geq 0})$ is equivalent to $\Mod H^0(\A)$ and the cohomological functor $H^0$ sends a dg module $M$ to $H^0(M):A\mapsto H^0(M(A))$.
\item In particular $H^0$ is fully faithful on the full subcategory of quasi-representable dg modules and $H^0(M)$ is projective in the heart for each quasi-representable dg module $M$.
\end{itemize}

\begin{lemma}[{\cite[Lemma 3.6]{Chen24b}}]
A 3-term h-complex
\[
\begin{tikzcd}
A\ar[r,"f"]\ar[rr,bend right=8ex,"h"swap]&B\ar[r,"j"]&C
\end{tikzcd}
\]
is left exact if and only if its totalization $N$ lies in $\mathcal H$. In this case, the totalization $N$ is the cokernel of $H^0(B)\xrightarrow{H^0(j)}H^0(C)$ in $\mathcal H$.
\end{lemma}
The class of 3-term h-complexes that we are considering are the class of homotopy left exact sequences. Moreover, since $\A$ is left stable, each morphism $j:B\rightarrow C$ admits a homotopy kernel. It follows that totalizations of homotopy left exact sequences are exactly finitely presented right $H^0(\A)$-modules.

By Proposition~\ref{thm:main}, the category $H^0(\A)$ carries a canonical left triangulated structure. In particular, it admits weak kernels. So, the category $\mod H^0(\A)$ of finitely presented right $H^0(\A)$-modules is {\em wide} in $\Mod H^0(\A)$, i.e.~it is stable under extensions, kernels and cokernels.

\begin{notation}\label{not:N}
We denote by $\tilde{\mathcal N}\subseteq \mathcal H$ the essential image of $\mathcal N$ under the homological functor $H^0:\D(\A)\rightarrow \mathcal H$.
\end{notation}
\begin{corollary}\label{cor:N}
An object in $\mathcal H\simeq \Mod H^0(\A)$ lies in $\tilde{\mathcal N}$ if and only if it is finitely presented.
The t-structure $(\D(\A)^{\leq 0}, \D(\A)^{\geq 0})$ on $\D(\A)$ restricts to a bounded t-structure on $\mathcal N$.
\end{corollary}
\begin{proof}
Let $\mathcal N'$ be the full subcategory of $\mathcal N$ consisting of objects $N$ such
that $H^i(N)$ is finitely presented for each $i \in \mathbb Z$ and only finitely many of them are nonzero. Since $\mod H^0(\A)$ is a wide subcategory of $\Mod H^0(\A)$ , $\mathcal N'$ is a triangulated subcategory of $\mathcal N$. It certainly contains totalizations of
homotopy left exact sequences and thus is equal to $\mathcal N$. Therefore the objects $N$ in $\mathcal N$ have the property that
$H^i(N)$ is finitely presented for each $i \in \mathbb Z$ and only finitely many of them are nonzero.
\end{proof}
\begin{proof}[Proof of Theorem~\ref{thm:universal}]

The functor $H^0(F):H^0(\A)\rightarrow \D^b(\A)$ is a triangle functor from the left triangulated category $H^0(\A)$ to the triangulated category $\D^b(\A)$. By the universal property of the stablization, there is a triangle functor $G$, unique up to isomorphism, from $\S(H^0(\A))$ to $\D^b(\A)$ which makes the following diagram commute
\[
\begin{tikzcd}
H^0(\A)\ar[d]\ar[r,"H^0(F)"]&\D^b(\A)\\
\S(H^0(\A)).\ar[ru,dashed,"G"{swap}]&
\end{tikzcd}
\]
Let us first show that the triangle functor $G$ sends nonzero objects in $\S(H^0(\A))$ to nonzero objects in $\D^b(\A)$.
The objects in $\S(H^0(\A))$ are shifts of the objects in $H^0(\A)$.
Therefore, we only need to show that if the functor $G$ sends an object from $H^0(\A)$ to a zero object in $\D^b(\A)$, then it is zero in $\S(H^0(\A))$.
Note that by Corollary~\ref{cor:N}, the triangulated subcategory $\N$ of $\tr(\A)$ admits a bounded t-structure. Thus, it is Karoubian, i.e.~idempotent-split, by~\cite{ChenLe07}.
Let $A$ be an object in $H^0(\A)$.
Suppose that it is sent to a zero object in $\D^b(\A)$ under the functor $H^0(F)$.
Then $A$ lies in $\N$ when viewed as an object in $\tr(\A)$.
It follows that $\Hom_{\tr(\A)}(A',\Sigma^{-n}A)=0$ for any object $A'$ in $\A$ and $n>>0$.
We have $\Hom_{\tr(\A)}(A',\Sigma^{-n}A)\iso \Hom_{\tr(\A)}(A',\Omega^n (A))$.
Therefore $\Omega^{n} (A)$ is a zero object in $H^0(\A)$ and hence $A$ is zero in $\S(H^0(\A))$.

Next, we show that the functor $G$ is full.
Let $A$ and $B$ be two objects in $H^0(\A)$.
Suppose we are given a morphism $b/s$ in $\D^b(\A)=\tr(\A)/\mathcal N$
\[
\begin{tikzcd}
&W\ar[ld, Rightarrow,"s"{swap}]\ar[rd,"b"]&\\
A&&B,
\end{tikzcd}
\]
where $M\coloneqq \Cone(s)\in \N$.
Since $\Hom_{\tr(\A)}(\Sigma^{-1}A, \N^{\geq 2})=0$, we can assume that $M\in \N^{\leq 0}$.
Suppose that $M\in \N^{[-n,0]}$ for some $n\geq 0$.
We have the following in $\tr(\A)$
\[
\begin{tikzcd}
\Sigma^{-2}M\ar[d,equal]\ar[r,dashed]&W_1'\ar[d,dashed]\ar[r,dashed,"s_{1}'"]&\Omega(A)\ar[d]\ar[r]&\Sigma^{-1}M\ar[d,equal]\\
\Sigma^{-2}M\ar[r]&\Sigma^{-1}W\ar[r,"\Sigma^{-1}(s)"]\ar[d,dashed]&\Sigma^{-1}A\ar[r]\ar[d]&\Sigma^{-1}M\\
&\Sigma^{-1}N'\ar[r,equal]&\Sigma^{-1}N'.&
\end{tikzcd}
\]
So we obtain the following morphism in $\D^b(\A)$
\[
\begin{tikzcd}
&W_1'\ar[ld,Rightarrow,"s_1'"{swap}]\ar[rd]&\\
\Omega(A)&&\Sigma^{-1}B.
\end{tikzcd}
\]
As above, we can assume that $\Cone(s')\in \N^{[-n+1,0]}$.
After composing with the morphism $\Omega(B)\rightarrow \Sigma^{-1}B$, we obtain a morphism in $\D^b(\A)$
\[
\begin{tikzcd}
&W_2\ar[ld,Rightarrow,"s_2"{swap}]\ar[rd]&\\
\Omega(A)&&\Omega(B).
\end{tikzcd}
\]
Similarly as above, we can assume that $\Cone(s_2)\in \N^{[-n+1,0]}$.
We iterate the above procedure and 
we see that the corresponding morphism in $\D^b(\A)$
\[
\begin{tikzcd}
&W_{n+1}\ar[ld,Rightarrow,"s_{n+1}"{swap}]\ar[rd]&\\
\Omega^n(A)&&\Omega^n(B)
\end{tikzcd}
\]
has the property that $\Cone(s_{n+1})$ is in the heart of the t-structure on $\N$.
We have $\Hom_{\tr(\A)}(\Omega^{n+1}(A), \Sigma^{-1}N)=0$ for any $N$ in the heart of the t-structure on $\N$. 
Therefore, the map $s_{n+1}':W_{n+1}'\rightarrow \Omega^{n+1}(A)$ is a split epimorphism. 
It is now straightforward to see that  the morphism $b/s$ in $\D^b(\A)$ is from a morphism $\Omega^{n+1}(\A)\rightarrow \Omega^{n+1}(B)$ under the functor $G$. 
Therefore, the functor $G$ is full. 

As a result, the triangle functor $G$ is fully faithful and hence an equivalence of triangulated categories.
\end{proof}
\begin{example}
Let $\A$ be an additive category with kernels, regarded as a dg category concentrated in degree zero. 
In this setting, $\A$ is a left stable dg category, where the loop functor $\Omega$ coincides with the zero functor.
Consequently, the bounded derived category $\D^b(\A)$ vanishes, and so does the stablization of $\A$. 
\end{example}
\begin{example}
Let $\T_{\dg}$ be a pretriangulated dg category and $\A'\subset \T_{\dg}$ a full dg subcategory such that $H^0(\A)$ is an aisle in $H^0(\T_{\dg})$.
Put $\A=\tau_{\leq0}\A'$.
Then $\A$ is a connective dg category with homotopy kernels.
For any morphism $f:Y\rightarrow Z$ in $Z^0(\A)$, consider its homotopy kernel in $\T_{\dg}$ 
\[
\begin{tikzcd}
X\ar[r]\ar[rr,bend right=8ex]&Y\ar[r,"f"]&Z.
\end{tikzcd}
\]
The precomposition with the canonical map $\tau_{\leq 0}X\rightarrow X$ then induces the homotopy kernel of $f$ in $\A$
\[
\begin{tikzcd}
\tau_{\leq 0}X\ar[r]\ar[rr,bend right=8ex]&Y\ar[r,"f"]&Z.
\end{tikzcd}
\]
It follows that for any object $Z$ in $\A$, the desuspension $\Omega(Z)$ is homotopy equivalent to $\tau_{\leq 0}\Sigma^{-1}Z$.
If the associated t-structure is bounded, then the functor $\Omega$ is nilpotent. In this case, the bounded derived category $\D^b(\A)$ vanishes, and so does the stablization of $H^0(\A)$.
\end{example}
\section{Exact dg categories with enough projectives}
With the notion of left stable dg categories at hand, we could state the results in \cite[6.37]{Chen23} more precisely.
Let $(\B,\mathcal S)$ be a connective exact dg cateogry with enough projectives. We denote by $\P$ the full dg subcategory of $\B$ consisting of the projectives. Our main aim in this subsection is to prove the following.
\begin{theorem}\label{thm:dgquotient}
The dg quotient $\B/\P$ is left stable.
\end{theorem}
\begin{proof}
We first show that the dg quotient functor $\pi:\B\rightarrow\B/\P$ sends conflations to homotopy left exact sequences.
Let $X$ be a conflation in $\B$
\[
\begin{tikzcd}
A\ar[r,"f"]\ar[rr,bend right=8ex,"h"swap]&B\ar[r,"j"]&C.
\end{tikzcd}
\]
We denote by $N$ its totalization. Put $V=\Sigma^{-1}\Cone(j)$.
We have $\Hom_{\tr(\B)}(\P, \Sigma^{i}N)=0$ for $i\in \mathbb Z$.
Then for $A'\in \B$ and $i\in \mathbb Z$, we have $\Hom_{\tr(\B)}(A',\Sigma^i N)\iso \Hom_{\tr(\B)/\tr(\P)}(A',\Sigma^i N)$.
Since $X$ is a conflation in $\B$, we have $\Hom_{\tr(\B)}(A',\Sigma^i N)=0$ for $A'\in \B$ and $i\leq -1$.
Hence we have $\Hom_{\tr(\B)/\tr(\P)}(A',\Sigma^i N)=0$ for $A'\in \B$ and $i\leq -1$.
So we have $\Hom_{\tr(\B)/\tr(\P)}(A',\Sigma^{i}A)\iso \Hom_{\tr(\B)/\tr(\P)}(A',\Sigma^{i}V)$ for $A'\in \B$ and $i\leq 0$. It follows that the image of $X$ in $\B/\P$ is homotopy left exact.

We have $H^0(\B/\P)=H^0(\B)/[\P]$. Let $g:B\rightarrow C$ be a morphism in $Z^0(\B/\P)$. Then the homotopy equivalence class $\overline{g}$ of $g$ is a morphism in $H^0(\B)/[\P]$ represented by a morphism $\overline{h}$ in $H^0(\B)$ with $h$ a morphism in $Z^0(\B)$.
So $\overline{\pi(h)}$ is isomorphic to $\overline{g}$ as objects in the morphism category $\Mor(H^0(\B/\P))$. So $g$ is isomorphic to $\pi(h)$ in $H^0\Mor(\B/\P)$.
Hence, we may assume that $g$ is given by a morphism in $Z^0(\B)$.
Since $\B$ has enough projectives, we have the following diagram in $\B$
\[
\begin{tikzcd}
E\ar[r,tail]\ar[d,equal]&D\ar[r,dashed]\ar[d,dashed]&B\ar[d,"g"]\\
E\ar[r,tail,"r"swap]&P\ar[r, two heads,"s"swap]&C.
\end{tikzcd}
\]
We obtain a conflation 
\[
\begin{tikzcd}
D\ar[r,tail]&B\oplus P\ar[r,two heads]&C.
\end{tikzcd}
\]
Since $B\oplus P\rightarrow C$ is isomorphic to $g:B\rightarrow C$ in $H^0(\Mor(\B/\P))$, we deduce that the morphism $g:B\rightarrow C$ admits a homotopy kernel in $\B/\P$. This finishes the proof that $\B/\P$ admits homotopy kernels and is left stable.
\end{proof}
\begin{remark}
Let $\B$ be a connective exact dg category with enough projectives. For example, we could simply take $\B$ be the module category $\mod\mbox{-}\Lambda$ of finitely generated modules over an artin algebra $\Lambda$. In summary of the above results, we have the following diagram
\[
\begin{tikzcd}
\B\ar[r]\ar[d,rightsquigarrow]&\B/\P\ar[r]\ar[d,rightsquigarrow]&\D^b_{\dg}(\B/\P)\ar[d,rightsquigarrow]& \D^b_{\dg}(\B)/\pretr(\P)\ar[d,rightsquigarrow]\ar[l,"\sim"{swap}] \\
H^0(\B)\ar[r]&H^0(\B)/[\P]\ar[r]&\S(H^0(\B)/[\P]) & \D^b(\B)/\tr(\P)\ar[l,"\sim"{swap}],
\end{tikzcd}
\]
where the arrows $\leadsto$ means that the upper categories provide dg enhancements for the lower categories.
\end{remark}

\section*{Acknowledgements}
The author thanks Xiao-Wu Chen for helpful comments and discussions.

	\def\cprime{$'$} \def\cprime{$'$}
	\providecommand{\bysame}{\leavevmode\hbox to3em{\hrulefill}\thinspace}
	\providecommand{\MR}{\relax\ifhmode\unskip\space\fi MR }
	\providecommand{\MRhref}[2]{%
		\href{http://www.ams.org/mathscinet-getitem?mr=#1}{#2}
	}
	\providecommand{\href}[2]{#2}
%


	\bibliographystyle{amsplain}
	\bibliography{stanKeller}

@article {Chen26,
    AUTHOR = {Chen, Xiaofa},
     TITLE = {Exact dg categories {I}: {F}oundations},
   JOURNAL = {Adv. Math.},
  FJOURNAL = {Advances in Mathematics},
    VOLUME = {489},
      YEAR = {2026},
     PAGES = {Paper No. 110809},
      ISSN = {0001-8708,1090-2082},
   MRCLASS = {18G35 (16E45 16G10 18G80 18N40)},
       DOI = {10.1016/j.aim.2026.110809},
       URL = {https://doi.org/10.1016/j.aim.2026.110809},
}

@unpublished{Chen24b,
author={Chen, Xiaofa},
year={2024},
title={Exact dg categories {II} : {T}he embedding theorem},
note={arXiv:2406.11226 [math.RT]}
}

@unpublished{Chen23,
Year={2023},
Author={Chen, Xiaofa},
Title={On exact dg categories},
Note={arXiv:2306.08231 [math.RT]}
}

@unpublished{Mochizuki25a,
year={2025},
author={Mochizuki, Nao},
title={Higher-dimensional generalization of abelian categories via {DG}-categories},
note={arxiv:2501.06955 [math.CT]}
}

@article {ChenLe07,
    AUTHOR = {Le, Jue and Chen, Xiao-Wu},
     TITLE = {Karoubianness of a triangulated category},
   JOURNAL = {J. Algebra},
  FJOURNAL = {Journal of Algebra},
    VOLUME = {310},
      YEAR = {2007},
    NUMBER = {1},
     PAGES = {452--457},
      ISSN = {0021-8693,1090-266X},
   MRCLASS = {18E30},
  MRNUMBER = {2307804},
       DOI = {10.1016/j.jalgebra.2006.11.027},
       URL = {https://doi.org/10.1016/j.jalgebra.2006.11.027},
}

@preamble{"\def\cprime{$'$} "}

@article{Kelly65,
    Author = {Kelly, G. M.},
    Journal = {Proc. Cambridge Philos. Soc.},
    Pages = {847--854},
    Title = {Chain maps inducing zero homology maps},
    Volume = {61},
    Year = {1965}}

@article {Tabuada10,
    AUTHOR = {Tabuada, Gon\c{c}alo},
     TITLE = {On {D}rinfeld's dg quotient},
   JOURNAL = {J. Algebra},
  FJOURNAL = {Journal of Algebra},
    VOLUME = {323},
      YEAR = {2010},
    NUMBER = {5},
     PAGES = {1226--1240},
      ISSN = {0021-8693},
   MRCLASS = {16E45 (14A22 16S38)}
  }

\end{document}